\documentclass[preprint]{elsarticle}
\usepackage[left=2cm, top=2cm, bottom=.75cm,right=2cm]{geometry}

\usepackage{amssymb,amsmath}
 \usepackage{amsthm}

\usepackage{graphicx,subfigure}

\newcommand{\lm}{\lambda}
\newcommand{\e}{\varepsilon}
\newcommand{\al}{\alpha}

\newcommand{\be}{\begin{equation}}

\newcommand{\ee}{\end{equation}}
\newcommand{\bq}{\begin{eqnarray}}
\newcommand{\eq}{\end{eqnarray}}
\newcommand{\ex}{\mathbb{E}}

\newcommand{\nind}{\noindent}
\newcommand{\nn}{\nonumber}
\newcommand{\lb}{\lbrace}
\newcommand{\rb}{\rbrace}
\newcommand{\pr}{\mathbb{P}}
\newcommand{\R}{\mathbb{R}}
\newcommand{\C}{\mathbb{C}}
\newcommand{\mc}{\mathcal}

 \newtheorem{thm}{Theorem}[section]
    \newtheorem{prop}[thm]{Proposition}

    \newtheorem{lmma}[thm]{Lemma}

        \newtheorem{rem}[thm]{Remark}

\journal{}

\begin{document}

\begin{frontmatter}



\title{Large deviations for the boundary local time of doubly reflected Brownian motion\tnoteref{Thx}}
\tnotetext[Thx]{The authors would like to thank Chris Rogers for his initial suggestions on how to approach this problem and Paavo Salmimen.}

\author[MF]{Martin~Forde}
\ead{martin.forde@kcl.ac.uk}
\address[MF]{Dept. Mathematics,
King's College London, London WC2R 2LS}

\author[RK]{Rohini Kumar}
\ead{rkumar@math.wayne.edu}
\address[RK]{Dept. Mathematics,
Wayne State University,
Detroit, MI 48202. }

\author[HZ]{Hongzhong Zhang}
\ead{hzhang@stat.columbia.edu}
\address[HZ]{Dept. Statistics,  Columbia University,
  New York, NY 10027.}
\begin{abstract}
We compute a closed-form expression for the moment generating function $\hat{f}(x;\lm,\al)=\frac1 \lambda\mathbb{E}_x(e^{\alpha L_{\tau}})$, where $L_t$ is the local time at zero for standard Brownian motion with reflecting barriers at $0$ and $b$, and $\tau \sim \mathrm{Exp}(\lm)$ is independent of $W$. By analyzing how and where $\hat{f}(x;\cdot,\al)$ blows up in $\lm$, a large-time large deviation principle (LDP) for $L_t/t$ is established using a Tauberian result and the G\"{a}rtner-Ellis Theorem.
\end{abstract}

\begin{keyword}
Brownian motion \sep Large deviation \sep Local time.


\end{keyword}

\end{frontmatter}


\section{Introduction}
Diffusion processes with reflecting barriers have found many applications in finance, economics, biology, queueing theory, and electrical engineering.  In a financial context, we recall the currency exchange rate target-zone models in \cite{KRU91} (see also \cite{SVE91, BER92, DJ94}, and \cite{BAL98}), where the exchange rate is allowed the float within two barriers; asset pricing models with price caps (see \cite{HAN99}); interest rate models with targeting by the monetary authority (e.g.\cite{FAR03}); short rate models with reflection at zero (e.g. \cite{GOL97, GOR04}); and stochastic volatility models (most notably the Heston and Sch\"{o}bel-Zhu models).  In queueing theory, diffusions with reflecting barriers arise as heavy-traffic approximations of queueing systems and reflected Brownian motions is ubiquitous in queueing models \cite{HAR85, ABA87a, ABA87b}.  More recently, reflected Ornstein-Uhlenbeck(OU) and reflected affine processes have been studied as approximations of queueing systems with reneging or balking \cite{WAR03a, WAR03b}.  Applications of reflected OU processes in mathematical biology are discussed in \cite{RIC87}.

  Doubly reflected Brownian motion also arises naturally in the solution for the optimal trading strategy in the large-time limit for an investor
  who is permitted to trade a safe and a risky asset under the Black-Scholes model, subject to proportional transaction costs with exponential or power utility (see \cite{GM13} and \cite{GGMS12} respectively).

The asymptotics in this article are obtained using a Tauberian theorem.  Tauberian results typically allow us to deduce the large-time or tail behavior of a quantity of interest based on the behavior of its Laplace transform (see Feller\cite{Fel71} or the excellent monograph of Bingham et al.\cite{BGT87} for details).  In situations where precise tail asymptotics for a distribution function are unknown but a moment generating function is available in closed-form, \cite{BF08} establish sharp tail asymptotics on logarithmic scale as applications of the standard Kohlbecker and Karamata Tauberian theorems, combined with an Esscher change of measure; their
methodology is applied to various time-changed exponential L\'{e}vy models (specifically the Variance Gamma model under a Gamma-OU clock and the Normal Inverse Gaussian model with a CIR clock) and the well known Heston stochastic volatility model.

In this article, we compute a closed-form expression for the moment generating function (mgf) $\hat{f}(x;\lambda,\al)=\mathbb{E}_x(e^{\alpha L_{\tau}})$, where $L_t$ is the local time at zero for standard Brownian motion with reflecting barriers at $0$ and $b$, and $\tau$ is an independent exponential random variable with parameter $\lm$. We do this by first deriving the relevant ODE and boundary conditions for $\hat{f}(x;\lm,\al)$ using an augmented filtration and computing the optional projection, and we then solve this ODE in closed form.  $\hat{f}(x;\lm,\al)$ does not appear amenable to Laplace inversion; however from an analysis of the location of the pole of $\hat{f}(x;\cdot,\al)$, we can compute the re-scaled log mgf limit $V(\al)=\lim_{t \to \infty}\frac{1}{t}\log \mathbb{E}_x(e^{\al L_t})$ for $\al \in \mathbb{R}$ using the Tauberian result in Proposition 4.3 in \cite{Kor02} via the so-called Fej\'{e}r kernel.  From this we then establish a large deviation principle for $L_t/t$ as $t \to \infty$ using the G\"{a}rtner-Ellis Theorem from large deviations theory,


Throughout the paper, we let $\pr_x(\cdot)=\pr(\cdot|X_0=x)$ denote the law of $X$ given its initial value at time 0 for any $x\in[0,b]$, and by $\ex_x(\cdot)$ the expectation under $\pr_x$. Further, we let $\ex\equiv\ex_0$.

\section{The modelling set up}

We begin by defining the Brownian motion $X$ with two reflecting boundaries.  Let $W_t$ be standard Brownian motion starting at 0. Then for any $x\in[0,b]$, there is a unique pair of non-decreasing,
continuous adapted processes $(L,U)$, starting at 0, such that
\bq
X_t= x+W_t \,+\, L_t \,-\,U_t \quad \in  \,\,\, [0,b], \quad \quad \forall t \ge 0. \nn
\eq
such that $L$ can only increase when $X=0$ and $U_t$ can only increase when $X=b$.  Existence and uniqueness follow easily from the more general work of Lions\&Sznitman\cite{LS84} the earlier work of Skorokhod\cite{Sko62}, or a bare-hands proof can be given by successive applications of the standard one-sided reflection mapping using a sequence of stopping times (see \cite{Wil92}.)

It can be shown that
\bq
 \lim_{t \to \infty} L_t/t =  \mathbb{E}(L_{\tau^b+\tau'})/\mathbb{E}(\tau^b+\tau'),&& \lim_{t \to \infty} U_t/t =  \mathbb{E}(U_{\tau^b+\tau'})/\mathbb{E}(\tau^b+\tau'), \nn \\
\lim_{t \to \infty} \frac{1}{t} \mathrm{Var}(L_t) = \sigma_L^2,  \quad && \quad\lim_{t \to \infty} \frac{1}{t} \mathrm{Var}(U_t) = \sigma_U^2, \nn
\eq
where $\tau^b=\inf \lb t: X_t=b \rb$, $\tau'=\inf \lb t \ge \tau^b: X_t=0 \rb$(see \cite{Wil92}) for some non-negative constants $\sigma_L,\sigma_U$.

\begin{prop}\label{prop: ODE}
Let $\tau$ denote an independent exponential random variable with parameter $\lm$.  Then for $\al<0$,
\bq
\hat{f}(x)\equiv\hat{f}(x; \lm,\al):= \frac{1}{\lm}\mathbb{E}_{x}(e^{\al L_{\tau}}) = \frac{1}{\lm}\int_0^{\infty} e^{-\lm t}\, \mathbb{E}_x(e^{\al L_t})dt \nn
\eq
is smooth on $(0,b)$ and satisfies the following  ODE
\bq
\label{eq:ODE}
\frac{1}{2}\hat{f}_{xx} = \lm \hat{f} \,-\, 1, \, \hat{f}_x(0)+\al\hat{f}(0)=\hat{f}_x(b)=0.
\eq
\end{prop}

\begin{proof}
We first
show that $\hat{f} \in C^{\infty}(0,b)$.  To this end, note that for $x \in [0,b]$,
\bq
\mathbb{E}_{x}(e^{\al L_{\tau}}) = \mathbb{P}_{x}(\tau>H_0)\, \mathbb{E}_{0}(e^{\al L_{\tau}}) \,+\,  \mathbb{P}_{x}(\tau \le H_0) \, \nn
\eq
where $H_x=\inf \lb t: X_t=x\rb$ is the first hitting time to $x$.  The law of $(b-X_t; t\in[0,H_0])$ given $X_t=x$ is the same as that of $(|W_t|; t\in[0,H_b])$ given $|W_0|=b-x$. Thus by Eq. 2.0.1 on page 355 of \cite{BS02} we have
\be
\pr_x(\tau>H_0)=\ex_x(e^{-\lm H_0})=\frac{\cosh((b-x)\sqrt{2\lm})}{\cosh(b\sqrt{2\lm})}\nn.
\ee
It follows that
\bq
\ex_x(e^{\al L_\tau})=\frac{\cosh((b-x)\sqrt{2\lm})}{\cosh(b\sqrt{2\lm})}\,[\ex_0(e^{\al L_\tau})-1]\,+\,1\,.\nn
\eq
That is,
\bq\label{eq:f}
\hat{f}(x)=\frac{\cosh((b-x)\sqrt{2\lm})}{\cosh(b\sqrt{2\lm})}\,(\hat{f}(0)-\frac{1}{\lm})\,+\,\frac{1}{\lm}\,,\,\,\,\forall x\in[0,b]\,.
\eq
It can then be easily seen from (\ref{eq:f}) that $\hat{f}\in C^\infty(0,b)$.

To show that $\hat{f}$ satisfies (\ref{eq:ODE}) and the boundary conditions, we construct a martingale that is adapted to the filtration generated by $X$.  More specifically, we introduce the natural filtration $\mathcal{F}_t=\sigma(X_s; s\le t)$ and the augmented filtration $\overline{\mathcal{F}}_t=\mathcal{F}_t\vee\sigma(\mathbf{1}_{\{\tau<t\}})$, where $\sigma(\mathbf{1}_{\{\tau<t\}})$ is the sigma algebra generated by $\mathbf{1}_{\{\tau<t\}}$. Then we have a uniformly bounded, and hence uniformly integrable $\overline{\mathcal{F}}_t$-martingale:
\bq
\overline{M}_t:=\ex(e^{\al L_\tau}|\overline{\mathcal{F}}_t)=\mathbf{1}_{\{\tau<t\}}e^{\al L_\tau}+\mathbf{1}_{\{\tau\ge t\}}e^{\al L_t}\ex_{X_t}(e^{\al L_\tau}) = \mathbf{1}_{\{\tau<t\}}e^{\al L_\tau}+\mathbf{1}_{\{\tau\ge t\}}e^{\al L_t}\lm\hat{f}(X_t)\, . \nn
\eq

We now define the optional projection of $\overline{M}_t$: using the fact that $X$ and $\tau$ are independent, we have
\bq
M_t = \ex(\overline{M}_t|\mathcal{F}_t) =\lm \int_0^t e^{\al L_s-\lm s}ds \,+ \,e^{\al L_t-\lm t}\lm\hat{f}(X_t)\, . \nn
\eq
Further, $M_t$ is a $\mathcal{F}_t$-martingale, in that for all $t>s$ we have
\bq
\ex(M_t|\mathcal{F}_s) =\ex(\ex(\overline{M_t}|\mathcal{F}_t)|\mathcal{F}_s) =\ex(\overline{M}_t|\mathcal{F}_s) =\ex(\ex(\overline{M}_t|\overline{\mathcal{F}}_s)|\mathcal{F}_s) =\ex(\overline{M}_s|\mathcal{F}_s) = M_s. \nn
\eq

Applying It\=o's lemma to $M_t$, we have that
\bq
dM_t = e^{\al L_t-\lm t}\left[\lm dt+\lm \hat{f}(X_t)(\al dL_t-\lm dt)+\frac{1}{2}\lm\hat{f}_{xx}(X_t)dt+\lm\hat{f}_x(X_t)(dW_t+dL_t-dU_t)\right]. \nn
\eq
But for $M_t$ to be a martingale, we must have
\bq
\frac{1}{2}\hat{f}_{xx}(x)-\lm \hat{f}(x)+1=0,\,\,\,\hat{f}_x(0)+\al \hat{f}(0)=0,\,\,\hat{f}_x(b)=0. \nn
\eq
This completes the proof.
\end{proof}

Solving the ODE in Proposition \ref{prop: ODE} we obtain the following result:

\begin{prop}\label{prop:f}
\bq
\label{eq:DesiredEq}
\hat{f}(x;\lm,\al)  =\frac{1}{\lm} + e^{x\sqrt{2 \lm} } A_\lm(\al) \,+\,
 e^{-x\sqrt{2 \lm} } B_\lm(\al)  \,
\eq
for $\lm >0,\al< 0$, where
\bq
A_{\lm}(\al)= \frac{\al e^{-b\sqrt{2\lm}}/\cosh(b\sqrt{2\lm})}{2\lm \,[
    \al^*(\lm)-\al]\,}, \, B_\lm(\al)= e^{2 \sqrt{2\lm}\, b} A_{\lm}(\al),\,
\al^*(\lm) = \sqrt{2\lm}\,\tanh(b\sqrt{2\lm}) .\label{eq:alphatstar}
\eq
\end{prop}

\begin{rem}\label{rem:0}
Observe that the expression for $\hat{f}(x)$ involves $\sqrt{\lm}$, which has a branch point at $\lm=0$. However, $\hat{f}$  remains a continuous function across the branch cut at $\lm=0$; thus $\hat{f}$ is an analytic function of $\lm$ in some punctured disc about $\lm=0$.
As $\lim_{\lm\to 0} \lm\cdot \hat{f}(\lm)=0$, we conclude that $\lm=0$ is a removable singularity.

\end{rem}

 \begin{rem}\label{rem:al}
It can be verified that $\alpha^*(\cdot)$ in \eqref{eq:alphatstar} is a strictly increasing  mapping from $[0,\infty)$ onto $[0,\infty)$. Further, we may analytically extend $\alpha^*(\cdot)$ to get a strictly increasing, strictly concave, smooth real-valued function that maps $(-\frac{\pi^2}{8b^2},\infty)$ onto $\mathbb{R}$.
\end{rem}

\section{Large-time asymptotics}

In this section, we characterize the large-time behaviour of $L_t$.
To this end,
let us consider the inverse of $\al^*$, $V(\al):=(\al^*)^{-1}(\al)$ for $\alpha\in \mathbb R$. From Remark \ref{rem:al}, we know that $V$ is a strictly increasing, strictly convex smooth function, with range $(-\frac{\pi^2}{8b^2},\infty)$.

\begin{lmma}
\label{lem:AE}
The equality \eqref{eq:DesiredEq} also holds for all $\al\in \R,\lm \in \mathbb{C}$ such that $\Re(\lm)>V(\al)$.
\end{lmma}
\begin{proof}
See \ref{AEproof}.
\end{proof}

\begin{prop}
\label{prop:LargeTmgf}
We have the following large-time behaviour for the moment generating function of $L_t$:
\bq
\lim_{t \to \infty}\frac{1}{t}\log \mathbb{E}_x(e^{\al L_t})= V(\al) <\infty\,\,\,\forall \al\in\R.\nn
\eq
\end{prop}
\begin{proof}
See \ref{LargeTmgfproof}.
\end{proof}

\begin{rem}
Note that $V(\cdot)$ does not depend on the starting value $x$, due to the ergodicity of $X$.
\end{rem}

The following lemma will be needed in the statement of the large deviation principle in the theorem that follows.

\begin{lmma}\label{lemma:rate_fn}
\begin{itemize}
\item[(a)] Define $
V^*(x):=\sup_{\al \in \mathbb{R}}[\al x-V(\al)]$ for all $x \ge 0$.  Then we have
\bq
\label{V*-formula}
V^*(x)=   \left\{
        \begin{array}{ll}
x \al^*(\lm^*)- \lm^* , \quad \quad\quad&\text{for }x>0\\
\pi^2/(8 b^2), &\text{for }x=0
        \end{array}\
        \right.,
        \eq
where $\lm^*=\lm^*(x)$ is the unique solution of $(\al^*)'(\lm)=1/x$ for fixed $x>0$.
\item[(b)] $V^*\in C([0,\infty))\cap C^1((0,\infty))$ and $V^*$ is a strictly convex function on $(0,\infty)$.
\item[(c)] $V^*$ attains its minimum value of zero uniquely at $x^*=\frac{1}{2b}$.
\end{itemize}
\end{lmma}
\begin{proof}
See \ref{rate_fnproof}.
\end{proof}

\begin{thm}
\label{thm:LDP}
 $L_t/t$ satisfies a large deviation principle on $[0,\infty)$ as $t \to \infty$ with a strictly convex rate function $V^*(x)$.
\end{thm}
\begin{proof}
From Lemma \ref{lemma:rate_fn}, we know that $V^*$ is a strictly convex function on $(0,\infty)$. Hence the set of exposed points of $V^*$ is $(0,\infty)$ (see Definition 2.3.3 in \cite{DZ98}), and since $D^0_V= (-\infty, \infty)$, the exposing hyperplane will always lie in $D^0_V$.  Therefore, by  the G\"{a}rtner-Ellis Theorem (see Theorem 2.3.6 in \cite{DZ98}), $L_t/t$ satisfies the LDP with convex rate function $V^*(x)$. \end{proof}





\appendix

\renewcommand{\theequation}{A-\arabic{equation}}
\setcounter{equation}{0}

\renewcommand{\theequation}{B-\arabic{equation}}
\setcounter{equation}{0}  

\section{Proof of Lemma \ref{lem:AE}}\label{AEproof}
Recall from Propositions \ref{prop: ODE} and \ref{prop:f} that,  for $\al<0$ and $\lm>0$,
\bq
\label{eq:Mlmal}
\int_0^{\infty} e^{-\lm t}\, \mathbb{E}_x(e^{\al L_t}) dt=\hat{f}(x;\lm,\al)=\frac{1}{\lm} \,+\, e^{x\sqrt{2 \lm} } A_{\lm}(\al) \,+\,
 e^{-x\sqrt{2 \lm} } B_\lm(\al) \,
\eq
 where $A_{\lm}(\al)= \frac{\al e^{-b\sqrt{2\lm}}/\cosh(b\sqrt{2\lm})}{2\lm \,[
    \al^*(\lm)-\al]\,}$ and
    $B_\lm(\al)= e^{2 \sqrt{2\lm}\, b} A_{\lm}(\al)$.  We wish to show that \eqref{eq:Mlmal} still holds for a wider range of $\al$ and $\lm$ values using analytic continuation.  We first note that $\hat{f}$ has a singularity when $\al=\al^*(\lm)$, and by Theorems 5a and 5b on page 57 in \cite{Wid46}, we know that the abscissa of convergence for a Laplace transform is a point of singularity and the Laplace transform is analytic in its region of convergence.

We are interested in the values of $\al\in \R$ and $\lm\in \C$ such that the Laplace transform $\hat{f}(x;\lm,\al)=\int_0^{\infty} e^{-\lm t}\, \mathbb{E}_x(e^{\al L_t}) dt$ is finite.  We recall the following fact: for any fixed $x\in[0,b]$,

\nind $(\dagger)$ $\hat{f}(x;\lm,\al)<\infty$ for $\al<0$ and $\lm>0$.

\nind We now proceed in three stages:

\begin{itemize}
\item
Fix $\lm>0$ (so $\lm \in \mathbb{R}$).  We apply the Widder results with $\al$ as the Laplace variable, i.e. we consider
 \bq
 \mathbb{E}(e^{\al L_{\tau}}) =\int_0^{\infty} e^{\al y} dF(y)\,,\nn
 \eq
 \nind where $F(y)$ is the distribution function of $L_{\tau}$.  By $(\dagger)$, the region of convergence is non-empty.  We can then extend the region of convergence up to $\al^*(\lm)>0$, as $\al^*(\lm)$ is the point of singularity.\\

\item
Fix $\al<0$. We apply the Widder results again, but we now take $\lm$ as the Laplace variable. By $(\dagger)$, the region of convergence is non-empty.  According to Widder, the abscissa of convergence (say $\lm_c$) is a point of singularity and $\hat{f}(x;\lm,\al)$ is analytic in $\lambda$ when $\Re(\lm)>\lm_c$.  So we are looking at a point of singularity on the real line, and this is the value of $\lm_c$ that satisfies $\al^*(\lm_c)=\al$.  Or, in other words, $\lm_c=(\al^*)^{-1}(\al)=V(\al)$. Thus, by Widder, $\hat{f}(x;\lm,\al)$ is finite when $\Re(\lm)>V(\al)$.\\

\item
Fix $\al>0$. We apply Widder's theorem using $\lm$ as the Laplace variable. By the first bullet point, we know that there exists some $\lm\in \R$ (such that $\al<\al^*(\lm)$),  for which $\hat{f}(x;\lm,\al)$ is finite. Hence, the region of convergence of $\hat{f}(x;\lm,\al)$ is non-empty for this $\al$.  Then, by Widder, the abscissa of convergence $\lm_c$  is a point of singularity and $\hat{f}(x;\lm,\al)$ is analytic for $\Re(\lm)>\lm_c$. The singularity is at $\al=\al^*(\lm)$. Solving for points of singularity on the real line i.e. solving for $\lm_c$ in $\al=\al^*(\lm_c)$,  gives us $\lm_c=V(\al)$ and so $\hat{f}(x;\lm,\al)$ converges when $\Re(\lm)>\lm_c=V(\al)$.\\

\end{itemize}

This gives the region of $\lm$ and $\al$ for which $\hat{f}(x;\lm,\al)$ converges: for every $\al\in \R$ and $\lm\in \C$ such that $\Re(\lm)>V(\al)$, and $\hat{f}(x;\lm,\al)$ is analytic in this region.


\renewcommand{\theequation}{C-\arabic{equation}}
\setcounter{equation}{0}  
\section{Proof of Proposition \ref{prop:LargeTmgf}}\label{LargeTmgfproof}

 From the known large-time behaviour of the local time of standard Brownian motion, we expect that $\ex_x(e^{\al L_t}) \sim const. \times \,e^{U(\al) t}$ as $t \to \infty$, for some non-decreasing function $U(\al)$ to be determined.  Then as $t\to\infty$,
\bq
\hat{f}(x;\lm,\al) =\int_0^{\infty} e^{-\lm t}\, \mathbb{E}_x(e^{\al L_t}) dt  \sim \int_0^{\infty} e^{-\lm t} \, const.\times e^{U(\al) t}dt,
\eq
and $\hat{f}(x;\lm,\al)$ blows up when $\lm=U(\al)$ (for $\al$ fixed).  But we know that $\hat{f}(x;\lm,\al)$ blows up at $\al=\al^*(\lm)$; thus we expect that $\lm=U(\al^*(\lm))$, i.e. $U(\al)=(\al^*)^{-1}(\al)=V(\al)$.  We now make this statement rigorous using a variant of Ikehara's Tauberian Theorem (see e.g. Theorem 17 on page 233 in Widder\cite{Wid46}).

We first define a positive  function $v$ on $\mathbb{R}$:
\[v(t)\equiv v(t;x,\alpha):=\mathbf{1}_{t\ge 0}\,e^{-V(\al)t}\ex_x(e^{\al L_t}).\]
Then the Laplace transform of $v$ is given by
\bq
\hat{v}(\lm)=\int_0^\infty e^{-\lm t}v(t)dt=\int_0^\infty e^{-(\lm+V(\al))t}\ex_x(e^{\al L_t})dt = \hat{f}_{\lm+V(\al),\al}(x)\,,\nn
\eq
which, by Lemma \ref{lem:AE} 
is analytic for all $\lm\in\mathbb{C}$ such that $\Re(\lm)>0$.
%
  We now need to characterize how $\hat{v}(\lm) $ blows up as $\Re(\lm) \downarrow 0$.  To this end, looking at the expression for $A_{\lm}(\al)$,
  we notice that $A_{\lm}(\al)$ has a pole at $\lm=V(\al)\in (-\frac{\pi^2}{8b^2},\infty)$, and is analytic elsewhere for $\Re(\lm)>-\frac{\pi^2}{8b^2}$ (see Remark \ref{rem:0}). It is also easily seen that, ${\al^*}'(\lm)>0$ for all $\lm\in(-\frac{\pi^2}{8b^2},\infty)$. Hence, by the Laurent expansion of $\hat{v}(\lm)$ at $0$,  there exists a function $g(\lm)$, which is analytic for all $\lm \in \mathbb{C}$ with $\Re(\lm)>-\e$ and $|\Im(\lm)|\le c$ for some constants $\e,c>0$, such that
  \be
  \hat{v}(\lm)=\frac{C}{\lm}\,+\,g(\lm)\,\nn
  \ee
  for some constant $C$ which we find to be positive ($C$ is the residue of $\hat{v}$ at $\lm=0$).  $g(x+iy)$ is continuous on $\mathcal{D}:=\{(x,y)\,:\, |x|\le\e, |y|\le c\}$, thus $g(x+iy)$ is  uniformly continuous on $\mc{D}$, so $g(x+iy)\to g(iy)$ uniformly as $x\downarrow 0$ for any fixed $y\in [-c,c]$.  Moreover, for any $x>0$
\bq
\int_{-c}^c |\hat{v}(x+iy)-\frac{C}{x+iy} - g(iy)|\,dy =  \int_{-c}^c |g(x+iy)-g(iy)|dy \nn
\eq
Since $g$ is analytic everywhere and uniformly continuous, if we take the limit as $x\to 0$, the above integral converges to $0$, so the function $g(x+i\cdot)$ also converge to $g(i\cdot)$ in $\mathbb{L}^1([-c,c])$, as $x\downarrow0$.

  We can now apply Proposition 4.3 in \cite{Kor02} to obtain that for the ``Fej\'er kernel'' $K(t)=\frac{1-\cos t}{\pi t^2}$,
  \bq\label{eq:asym}
  \lim_{t\to\infty}\int_{-\infty}^{ct}v(t-\frac{s}{c})\cdot  K(s)ds &=& C.
  \eq
  We now proceed as in the proof of Theorem 4.2 in \cite{Kor02} to show that $v(t)=O(1)$ as $t\to\infty$.

  \begin{enumerate}
  \item $\al>0$.  In this case we know that $\ex_x(e^{\al L_t})$ is non-decreasing, so $v(t)\ge v(s) e^{V(\al)(s-t)}$ for all $t\ge s\ge 0$. For any fixed $a>0$, using \eqref{eq:asym} we have that
  \begin{multline}
  C=\lim_{t\to\infty}\int_{-\infty}^{ct}v(t-\frac{s}{c})\cdot K(s)ds\ge \limsup_{t\to\infty}\int_{-a}^a v(t-\frac{s}{c})\cdot K(s)ds\ge\limsup_{t\to\infty} v(t-\frac{a}{c})\,e^{-2V(\al)\frac{a}{c}}\int_{-a}^a K(s)ds \,,\nn
  \end{multline}
  which implies that
  \bq
  \limsup_{t\to\infty}v(t) \le \frac{e^{2V(\al)\frac{a}{c}}}{\int_{-a}^aK(s)ds}\,C \,\,\,\,<\,\,\,\, \infty \,. \nn
  \eq
  Hence, there exists a constant $M>0$ such that $v(t)\le M$ for all $t$.
Similarly, for any fixed $a>0$, we have
  \begin{multline}
\liminf_{t\to\infty}v(t+\frac{a}{c})\,e^{2V(\al)\frac{a}{c}}\int_{-a}^a K(s)ds \ge \liminf_{t\to\infty}\int_{-a}^{a}v(t-\frac{s}{c}) K(s)ds\\
  =\liminf_{t\to\infty}\,(\int_{0}^{ct}\,+\,\int_{ct}^{\infty}\,-\, \int_{-\infty}^{-a}\,-\,\int_a^{\infty}) v(t-\frac{s}{c}) K(s)ds \ge\liminf_{t\to\infty}\,(\int_{0}^{ct}+\int_{ct}^{\infty})v(t-\frac{s}{c}) K(s)\nn \\
  -\limsup_{t\to\infty}\int_{-\infty}^{-a}\,v(t-\frac{s}{c}) K(s)\,-\,\limsup_{t\to\infty}\int_a^{\infty} v(t-\frac{s}{c}) K(s)\ge C\,-\,\frac{4M}{\pi}\int_a^\infty\frac{1}{s^2}ds  = C \,-\, \frac{4M}{\pi a}\,,\nn
  \end{multline}
  where we have used \eqref{eq:asym} and the fact that $0\le K(t)\le \frac{2}{\pi t^2}$ in the last inequality. Hence, for $a>0$ sufficiently large, we have
  \bq
  \liminf_{t\to\infty}v(t) \ge \frac{e^{-2V(\al)\frac{a}{c}}}{\int_{-a}^aK(s)ds}\,(C-4M/\pi a)\,\,\,\,> \,\,\,\, 0.\nn
  \eq
  \item $\al<0$. In this case we know that $\ex_x(e^{\al L_t})$ is non-increasing, so $v(t)\le v(s)e^{V(\al)(s-t)}$ for all $t\ge s\ge 0$. Using the same argument as above, we have, for any fixed $a>0$,
  \bq
  C e^{2V(\al)\frac{a}{c}}&\ge&\limsup_{t\to\infty}v(t+\frac{a}{c})\,\int_{-a}^aK(s)ds,\nn\\
  (C-\frac{4M}{\pi a}) \, e^{-2V(\al)\frac{a}{c}}&\le&\liminf_{t\to\infty}v(t-\frac{a}{c})\,\,\int_{-a}^aK(s)ds \,.\nn
  \eq
  Hence for $a>0$ sufficiently large, we have
  \bq
  0  \le \frac{e^{-2V(\al)\frac{a}{c}}}{\int_{-a}^aK(s)ds}\,(C-4M/\pi a) \le \liminf_{t\to\infty} v(t) \le \limsup_{t\to\infty}v(t) \le \frac{e^{2V(\al)\frac{a}{c}}}{\int_{-a}^aK(s)ds} \,C<\infty\, . \nn
  \eq
  \end{enumerate}
  Hence, by Proposition 4.3 in \cite{Kor02}, the result follows.\renewcommand{\theequation}{D-\arabic{equation}}
\setcounter{equation}{0}  

\section{Proof of Lemma \ref{lemma:rate_fn}}\label{rate_fnproof}

We break the proof into three parts:
\begin{itemize}
\item[(a)] Computing the Legendre transform of $V$ boils down to solving $V'(\al)=x$.  But this is the same as solving $(V^{-1})'(\lm)=\frac{1}{x}$ for $\lm$, when $x>0$. Recall that $V^{-1}(\cdot)=\al^*(\cdot)$ is known in closed form.   Since $(\al^*)''(\lm)<0$ for all $\lm$ (from Remark \ref{rem:al}), by the Inverse function theorem, $\lm^*(x):= ((\al^*)')^{-1}(1/x)$ is well-defined and  $\lm^*\in C^1((0,\infty))$.  Using the fact that $\al^*(\lm^*)=V^{-1}(\lm^*)$, we have
\bq
V^*(x)= x \al^* - V(\al^*)= x \al^*(\lm^*(x)) - \lm^*(x) \nn \,.
\eq

When $x=0$, the definition of $V^*$ in Lemma \ref{lemma:rate_fn} gives us $V^*(0)= \sup_{\al\in\mathbb{R}}\{-V(\al)\}=- \inf_{\al\in\mathbb{R}}\{V(\al)\}= - \lim_{\al\to-\infty}V(\al)= \pi^2/(8 b^2)$, where the last two equalities hold because  $V$ is a monotonically increasing function with range $(-\pi^2/(8 b^2),\infty)$.

\item[(b)] By the Inverse function theorem, we know that $\lm^*\in C^1((0,\infty))$ and so is $\al^*$, thus $V^*\in C^1((0,\infty))$. It is easy to check that $\lim_{x\downarrow 0}\{x \al^*(\lm^*(x))- \lm^*(x)\}=\pi^2/(8b^2)=V^*(0)$, which gives continuity of $V^*$ up to the boundary $x=0$. Using \eqref{V*-formula}, we obtain
\bq
(V^*)'(x)&=&\al^*(\lm^*(x))+x\cdot (\al^*)'(\lm^*(x))\cdot (\lm^*)'(x)-(\lm^*)'(x)\nn \\
&=&\al^*(\lm^*(x))+x\cdot \frac{1}{x}\cdot(\lm^*)'(x)-(\lm^*)'(x)=\al^*(\lm^*(x)) \nn \,
\eq
\nind $(V^*)'(x)=\al^*(\lm^*(x))$.  Thus we have (using again $(\al^*)''<0$)
\begin{align*}(V^*)''(x)&=(\al^*)'(\lm^*(x))\cdot (\lm^*)'(x)=\frac{1}{x}\cdot ((\al^*)'^{-1})'(\frac1 x)\cdot -\frac{1}{x^2}=-\frac{1}{x^3}\cdot \frac{1}{(\al^*)''(\lm^*(x))}>0\,.\end{align*}

\item[(c)]  Since $V^*$ is strictly convex, it has a unique minimum. The unique minimum of $V^*$ occurs at $x^*=((V^*)')^{-1}(0)=V'(0)=1/{\al^*}'(0)=\frac{1}{2b}$

\end{itemize}

\renewcommand{\theequation}{E-\arabic{equation}}
\setcounter{equation}{0}  

\bibliographystyle{elsarticle-num}







\end{document}